\documentclass{amsart}
\usepackage{amsmath,amsfonts}
\usepackage{mathtools}
\usepackage{amssymb,latexsym}
\usepackage{xcolor}
\definecolor{darkblue}{RGB}{30,30,200}
\usepackage{hyperref}
\hypersetup{
  colorlinks = true,
  urlcolor   = darkblue,
  linkcolor  = darkblue,
  citecolor  = darkblue
}
\usepackage{graphicx}
\usepackage[all]{xy}
\usepackage{layout}
\usepackage{tikz-cd}
\usetikzlibrary{matrix,arrows}

\newtheorem{theorem}{Theorem}[section] 
\newtheorem{lemma}[theorem]{Lemma}     
\newtheorem{corollary}[theorem]{Corollary}
\newtheorem{proposition}[theorem]{Proposition}

\theoremstyle{definition}

\newtheorem{definition}[theorem]{Definition}
\newtheorem{example}[theorem]{Example}

\newtheorem{remark}[theorem]{Remark}

\title[quasitriangular operator algebras]
{Quasitriangular Operator Algebras} 

\author[M.Amini, M. Moradi, I. Mousavi]{Massoud Amini, Mehdi Moradi, Ismaeil Mousavi}

 \address{
	Department of Pure Mathematics\\ Faculty of Mathematical Sciences\\
	Tarbiat Modares University\\
	Tehran\\ Iran}
	
	\email{mamini@modares.ac.ir\\ m-moradi@modares.ac.ir\\ ismaeilmousavi@gmail.com}



\begin{document}
\maketitle
\nonfrenchspacing

\begin{abstract}
We give characterizations of quasitriangular operator algebras along the line of Voiculescu's characterization of quasidiagonal $C^*$-algebras.
\end{abstract}



\section{Introduction}

Quasitriangular operators on a Hilbert space were
introduced and studied in the 60's by Paul Halmos in \cite{hal} (and later in \cite{hal2}). As Halmos also mentions, the notion was already used (and indeed played a central role) in the proof of the Aronszjn-Smith theorem \cite{as} on the existence of invariant subspaces for compact operators and its extensions (like that of Arveson and Feldman \cite{af}). These operators naturally generalize the notion of triangular matrices.

Triangular operator algebras were studied by John Ringrose \cite{r}, under the name of nest algebras, and discussed in the Ph.D. thesis of his student, Christopher Lance \cite{l}. The terminology was used first in a paper by Kadison and Singer \cite{ks}. There have been various characterizations of triangular subalgebras of specific classes of $C^*$-algebras, including that of Power for triangular subalgebras of AF-algebras \cite{p}.

The natural generalization of these to quasitriangular algebras were introduced by William Arveson \cite{ar} and studied by him and coauthors in \cite{fam}. All these studies are in the spirit of nest algebras. The first instance of a theory independent of the notion of nest algebras was presented by Dan Voiculescu in \cite{v}. Here quasidiagonal $C^*$-algebras are defined in terms of representations, something which is anticipated by the original ideas of Halmos (compare this with the approach of Bruce Wagner \cite{w} based on nest algebras). Further characterizations are given by Marius Dadarlat \cite{d}.

The main objective of this paper is to characterize quasitriangular (non selfadjoint) operator algebras in terms of their representations. Note that in the selfadjoint case (that is, for $C^*$-algebras) the  notions of quasidiagonality and quasitriangularity  coincide, and so such a characterization is a non selfadjoint version of the results of Voiculescu and Dadarlat. The first characterization relies on a result of Arveson on the extension of completely positive (c.p.) maps from operator systems to the enveloping $C^*$-algebra \cite{ar1} and the second is proved directly. We also investigate basic properties of quasitriangular operator algebras (like stable finiteness and existence of trace) and provide examples and non-examples. Our approach is based on finite dimensional approximation and we do not allude to the traditional approach based on nest algebras.

This paper could be regarded only as the starting point in comparing and contrasting quasitriangularity (of the non selfadjoint operator algebras) versus quasidiagonality (of selfadjoint ones). Although the two phenomena (which coincide in the selfadjoint case) show many similarities, the quasitriangularity has a pathological behavior which does not appear in quasidiagonality: while a diagonal (i.e., commutative) algebra is always quasidiagonal, the analogous thing does not happen in the non selfadjoint realm, namely,  a triangular algebra may fail to be quasitriangular. Indeed the Cuntz $C^*$-algebra $\mathcal O_n$ has a wealth of maximal triangular subalgebras which are not finite and so could not be quasitriangular. This phenomenon is due to the break of symmetry in the non selfadjoint case. Indeed, one has two options for a one way version of quasidiagonality (that is, upper or lower quasitriangularity) and choosing each of these is in the price of loosing the other. This might seem a bit annoying at the first glance, but we trust that it would be a source of diversity for quasitriangular algebras which is absent in the symmetric quasidiagonal case. Though this feature may prove to be useful and so calls for an independent research, here we could only sketch it and yet have no clear idea how one may use this potential in applications with an intrinsic lack of symmetry.

\section{Definitions and Results}

A Banach algebra $T$ is called {\it triangular approximate finite dimensional} (TAF) if it is a norm-closed subalgebra of a unital AF $C^*$-algebra $A$ such that $D:=T\cap T^*$ is a masa in $A$ which satisfies the Str\u{a}til\u{a}-Voiculescu condition \cite{sv}: there is a nested sequence $(A_k)$ of finite dimensional unital subalgebras of $A$ with a dense union such that $D_k:=D\cap A_k$ is a masa in $A_k$ \cite{ppw}. If $T$ is TAF then $T$ is the direct limit of the nested family $(T\cap A_k)$ \cite[Corollary 2.3]{ppw}, and conversely, for an increasing sequence of triangular algebras $T_k\subseteq A_k$ with diagonal $D_k$, the direct limit $T$ of $T_k$'s is TAF \cite[Theorem 2.6]{ppw}.

\begin{definition}\label{taf} An operator algebra $A$ is called TAF-embeddable if it is embedded completely isometrically into a TAF algebra.
\end{definition}

An operator algebra $A$ is called {\it residually finite dimensional} (RFD) if there is a family (possibly uncountable) of c.c. homomorphisms $\rho_i: A\to \mathbb M_{k(i)}(\mathbb C)$ which are asymptotically isometric, that is,  $\|a\|=\lim_i\|\rho_i(a)\|$, for each $a\in A$ (see \cite[Chapter 7]{bo}, for the same notion for $C^*$-algebras, where $\rho_i$'s are assumed to be $*$-homomorphisms). Next we need to define a triangular version of the above notion. By a triangular subalgebra of $\mathbb M_{k}$ we mean a subalgebra of the form $\mathbb T_{k_1}\oplus\cdots\oplus\mathbb T_{k_n}$, where $\mathbb T_{k_i}$ is the algebra of all upper triangular matrices in $\mathbb M_{k_i}$.

\begin{definition}\label{trfd} An operator algebra $A$ is called TRFD if it is RFD and the range of each homomorphism $\rho_i: A\to \mathbb M_{k(i)}(\mathbb C)$ in the above definition is a triangular subalgebra of $\mathbb M_{k_i}$.
\end{definition}

Two representations $\pi: A\to \mathbb B(H)$ and $\sigma: A\to \mathbb B(K)$ are {\it approximately unitarily equivalent} if there is a sequence $u_n: H\to K$ of unitary operators such that $\|\sigma(a)-u_n\pi(a)u_n^*\|\to 0,$ for each $a\in A$. Note that when both $\pi$ and $\sigma$ have ranges inside triangular operator algebras (in the sense of Kadison and Singer \cite{ks}, that is, the operator algebras $A$ of $\mathbb B(H)$ with $A\cap A^*=D$, where $D$ is the set of diagonal operators), then in each step, $u_n\pi(A)u_n^*$ is again a triangular operator algebra (since triangularity is preserved under unitary equivalence). This is along the general idea of Kadison and Singer to have non selfadjoint operator algebras whose elements could be simultaneously triangularized (c.f. \cite[page 245]{mss}, \cite[Theorem 3.2.1]{ks}).

Our first main result asserts that a separable TRFD operator algebras is TAF-embeddable if any pair of finite dimensional representations are approximately unitarily equivalent after adding an enough large multiple of one of them to both, providing a non selfadjoint analog of a result due to Dadarlat \cite{d2}.

\begin{theorem}\label{tafe}
Let $A$ be a separable TRFD operator algebra such that for each pair of finite dimensional representations $\sigma_i: A\to \mathbb M_k(\mathbb C)$, $i=1,2$, with triangular ranges, there is a positive integer $N\geq 1$ such that $(N+1)\sigma_1$ is approximately unitarily equivalent to $\sigma_2\oplus N\sigma_1$. Then $A$ is TAF-embeddable. 	
\end{theorem}

\begin{definition}\label{qt} An operator algebra $A$ is called quasitriangular (QT) if there is a net of completely contractive maps $\phi_i: A\to \mathbb M_{k_i}(\mathbb C)$ which are asymptotically isometric and asymptotically multiplicative, that is,  $\lim_i\|\phi_i(a)\|=\|a\|$ and $\lim_i\|\phi_i(ab)-\phi_i(a)\phi_i(b)\|=0$, for each $a,b\in A$. A collection of bounded operators $S\subseteq \mathbb B(H)$ is called quasitriangular (QT) if given $\varepsilon>0$ and finite subsets $\mathfrak F\subseteq S$ and $F\subseteq H$, there is a finite rank projection $P\in \mathbb B(H)$ with $\|PT-PTP\|<\varepsilon$ and $\|Pv-v\|<\varepsilon$, for each $T\in \mathfrak F$ and $v\in F$. Finally, a representation $\pi: A\to \mathbb B(H)$ of an operator algebra $A$ is called quasitriangular (QT) if $\pi(A)\subseteq \mathbb B(H)$  is  quasitriangular.
\end{definition}

A representation of an operator algebra $A$ is a completely contractive (c.c.) homomorphism $\pi: A\to \mathbb B(H)$, where $H$ is a Hilbert space. A faithful representation is assumed moreover to be completely isometric (c.i.). A representation $\pi: A\to \mathbb B(H)$ is called {\it essential} if $\pi(A)\cap \mathbb K(H)=\{0\}$.

The second main result of this paper gives a characterization of separable unital quasitriangular operator algebras in terms of their faithful representations. This is an extension of a result of Voiculescu which  asserts that a separable C*-algebra is QD iff it forms a QD family of operators in some concrete Hilbert space. Since the ``third isomorphism theorem'' fails in its most general form for operator algebras \cite{abr}, the QT version of this result for operator algebras satisfying an extra condition (here called {\it essentially embedded}, requiring that the quotient map onto the Calkin algebra induces a complete isometry on the given operator algebra; see the next section for details and examples).   

\begin{theorem}\label{qtr}
For a separable unital operator algebra $A$, consider the following assertions:

\ \ $(i)$ $A$ is QT,

\ \ $(ii)$ $A$ has a faithful QT representation on a separable Hilbert space,

\ \ $(iii)$ Every faithful unital essential representation of $A$ on a separable Hilbert space is QT.

Then $(iii)\Rightarrow (ii)\Rightarrow(i)$. If moreover $A$ is essentially embedded, then we also have $(i)\Rightarrow (iii)$.
\end{theorem}

The above theorem could be used to construct QT and non-QT operator algebras. It is shown by Halmos \cite{hal} that the unilateral shift $S$ on a separable Hilbert space is not QT (Douglas and Pearcy later constructed an operator $T$ on a separable Hilbert space so that neither $T$ nor $T^*$ is QT, answering a question of Halmos in negative \cite{dp}). Now the operator algebra generated by $S$ (which is the non selfadjoint Toeplitz algebra, here denoted by $\mathfrak T$) would not be QT by the above theorem. This has some consequences. For instance, if a separable unital operator algebra $A$ is QT, then $A$ could not contain a non surjective isometry (since otherwise it should contain a copy of non-QT algebra $\mathfrak T$), and neither does $A\otimes \mathbb M_n(\mathbb C)$, for each $n$. In other words, $A$ is {\it stably finite} (which simply means $\mathbb M_n(A)$ is {\it finite}--i.e., does not contain non unitary isometries--for each $n\geq 1$). We may get the same result for non-unital algebras by passing to the minimal unitization (which remains separable, if the original algebra is so). Hence we have the following result.

\begin{proposition}\label{fg}
An operator algebra is QT if and only if all of its finitely generated closed subalgebras
are QT.
\end{proposition} 

It immediately follows from the above proposition and Arveson Extension Theorem that the separability condition could be removed from Theorem \ref{qtr}. This allows one to repeat the above argument for the lack of non surjective isometries in the non separable case, leading to  the next result.

\begin{corollary}\label{sf}
A QT operator algebra $A$ is stably finite.
\end{corollary}

The above corollary could be employed to construct more examples of non-QT operator algebras. It is known that in the Cuntz $C^*$-algebra $\mathcal O_n$ we have an abundance of analytic subalgebras \cite{hp}. In particular, it contains a canonical UHF
subalgebra (such that each strongly maximal triangular subalgebra of the
UHF subalgebra has an extension to a strongly maximal triangular
subalgebra of $\mathcal O_n$). This shows that there are lots of triangular subalgebras in $\mathcal O_n$. One concrete example of such a subalgebra is the Volterra algebra (already studied by Power \cite{p1}), which is maximal triangular, but not strongly maximal triangular \cite[pages 29 \& 32]{hp}. However, this operator algebra is not QT, since it is generated by the Cuntz partial isometries it contains (and so is not finite). 

On the other hand, if the weight sequence of a bilateral weighted shift on a separable Hilbert space has 0 as a limit point (in both directions) then it is quasidiagonal \cite[Theorem 1]{s}. By almost the same proof as in \cite{s}, one could check that if the weight sequence of a unilateral weighted forward shift has 0 as a limit point then it is quasitriangular (take a subsequence of weights going to 0 and project on the the subspace generated by the basis elements with the same indices as the subsequence). This way we could construct a separable (singly generated) QT operator algebra (which is not QD if we choose the weights not to go to zero when indices are negated). 

Our next result asserts that that every QT operator algebra can be locally approximated by a residually finite dimensional (RFD) operator algebra. The proof is an
 adaptation of the Halmos original proof that every quasitriangular operator
can be written as a block triangular operator plus a compact (see also, \cite{he}). We call an operator algebra $A$  {\it block triangular} if there exists an increasing sequence of finite
rank projections $P_n$ going SOT to 1 such that $P_na=P_naP_n$, for all $a\in A$ and $n\geq 1$. As for $C^*$-algebras, it is easy to see that an operator algebra  is RFD if and only if it has a faithful representation with block triangular image (c.f. \cite{cm}, \cite{cr} for alternative approaches). 
Note that an operator algebra might be RFD while its enveloping C*-algebra is not so \cite[Example 3.4]{cm} and \cite[Example 1]{th}. 

If $\varepsilon>0$ and $F, A\subseteq \mathbb B(H)$ are sets of operators then  $F$ is $\varepsilon$-contained in $A$ if for each $x \in F$ there exists $y \in A$ with $\|x - y\| < \varepsilon$.
When this is the case we write $F\subseteq_{\varepsilon} B$.

\begin{theorem}\label{rfd}
A separable  operator algebra $A\subseteq \mathbb B(H)$ is QT (as a set of operators)  iff for every finite set $F\subseteq  A$ and every $\varepsilon > 0$,  there exists a
block triangular algebra $B \subseteq  \mathbb \mathbb B(H)$ such that $ F \subseteq_{\varepsilon} B$ and
$A + \mathbb K(H) = B + \mathbb K(H)$. If moreover  $A$ is  essentially embedded and $A\cap \mathbb K(H)=0$, then it is a quotient of $B$. 
\end{theorem}

A {\it state} $\tau$ on a unital operator algebra $A$ is a linear functional satisfying $\tau(1)=\|\tau\|=1$. It is called a  {\it tracial state} if moreover, $\tau(ab)=\tau(ba)$, for each $a,b\in A$. We have the following extension of a classical result of Voiculescu \cite[2.4]{v2}.

\begin{proposition}\label{tr}
Every unital QT operator algebra  has a 
tracial state.
\end{proposition}

\section{Proofs}

In this section we gives proofs of the results of the previous sections. Because the notion of positivity is not available for non selfadjoint operator algebras, in some cases we have to deviate from the line of the proof of the analogous result for the selfadjoint case, or allude to the basic idea of Arveson to get c.p. extensions of c.c. maps \cite{ar1}.

In the first lemma we deal with direct limit of operator algebras. The connecting maps are homomorphism, and unlike $*$-homomorphisms these are not guaranteed to be contractive (even continuous). Therefore, we only work with c.c. homomorphisms.

\begin{lemma}\label{lemma1}
Let $A$ be a separable unital operator algebra and $H$ be a separable Hilbert space. If there is an increasing sequence of finite dimensional triangular subalgebras $B_n\subseteq \mathbb B(H)$, constant non-negative integer $k\geq 1$, injective c.c. homomorphisms $\pi_n: B_n\to B_{n+k}$, and u.c.c. maps $\sigma_n: A\to B_{n}$ such that

$(i)$ the family $\{\sigma_n\}$ is asymptotically completely isometric and asymptotically multiplicative,

$(ii)$ $\sum_{n=1}^{\infty} \|\pi_n\circ\sigma_n(a)-\sigma_{n+k}(a)\|<\infty,$ for some subset $G\subseteq A$ with norm dense linear span and each $a\in G$, where the norm in the summand is calculated in $\mathbb B(H)$,

then $A$ is TAF-embeddable.
\end{lemma}

\begin{proof}
By a rearrangement of indices we may assume that $k=1$. Identify $B_n$ with a subalgebra of the direct limit $B$ of the system $\{(B_n, \pi_n)\}$ and regard each $\sigma_n$ as a map into $B$. Under this identification, one could rewrite $(ii)$ as $\sum_{n=1}^{\infty} \|\sigma_n(a)-\sigma_{n+1}(a)\|<\infty,$ which means that $\{\sigma_n(a)\}$ is Cauchy in $B$, for each $a\in G$. This shows that $\sigma: span(G)\to B;$ $\sigma(x):=\lim_{n}\sigma_n(x)$ is well defined. By the first assumption in $(i)$, the linear map $\sigma$ is completely isometric, and so has a complete isometric extension to $\bar\sigma: A\to B$, which is also a homomorphism, by the second assumption in $(i)$. Finally, observe that $B$ is TAF by \cite[Theorem 2.6]{ppw}.
\end{proof}

The proof of the first main result of this paper is a careful adaptation of an argument of M. Dadarlat \cite{d2} to the triangular setting.

\begin{proof}[Proof of Theorem~\ref{tafe}] Choose a asymptotically isometric sequence (by separability) of finite dimensional c.c. representations (by TRFD assumption) $\rho_i: A\to \mathbb M_{k(i)}(\mathbb C)$ such that
$$A_i:={\rm Im}(\rho_i)=\mathbb T_{k_1(i)}\oplus\cdots\oplus\mathbb T_{k_\ell(i)}\subseteq \mathbb M_{k(i)}(\mathbb C),$$
with $k_1(i)+\cdots k_\ell(i)=k(i)$, where $\ell$ depends on $i$. Put $B_1=A_1$ and $\sigma_1=\rho_1: A\to B_1\subseteq  \mathbb M_{k(1)}(\mathbb C)$. For finite dimensional representations
$$\sigma_1\otimes 1_{k(2)}=k(2)\sigma_1: A\to B_1\otimes\mathbb T_{k(2)}\subseteq \mathbb M_{k(1)k(2)}(\mathbb C),$$
and
$$1_{B_1}\otimes\rho_2: A\to B_1\otimes A_2\subseteq \mathbb M_{k(1)k(2)}(\mathbb C),$$
choose positive integer $N_1\geq 1$ such that
$$k(2)(N_1+1)\sigma_1: A\to B_1\otimes\mathbb T_{k(2)(N_1+1)}\subseteq \mathbb M_{n(1)}(\mathbb C),$$
and
$$(1_{B_1}\otimes\rho_2)\oplus k(2)N_1\sigma_1: A\to (B_1\otimes A_2)\oplus (B_1\otimes \mathbb T_{k(2)N_1})\subseteq \mathbb M_{n(1)}(\mathbb C),$$
are approximately unitarily equivalent, where $n(1):=k(1)k(2)(N_1+1)$.

Choose a countable dense subset $\{a_i\}$ in $A$ and for $a_1$, choose a unitary $u_1\in \mathbb M_{n(1)}(\mathbb C)$ with
$$\|u_1k(2)(N_1+1)\sigma_1(a_1)u_1^*-(1_{B_1}\otimes\rho_2(a_1))\oplus k(2)N_1\sigma_1(a_1)\|<\frac{1}{2}.$$
Put $C_1:=B_1$, $B_2:=u_1(C_1\otimes \mathbb T_{k(2)N_1})u_1^*\subseteq \mathbb M_{n(1)}(\mathbb C)$, and define $\pi_1: C_1\to B_2$ by $\pi_1(c):=u_1(c\otimes 1_{k(2)(N_1+1)})u_1^*$. Put $C_2:=(B_1\otimes A_2)\oplus (C_1\otimes \mathbb T_{k(2)N_1})\subseteq \mathbb M_{n(1)}(\mathbb C)$ and define $\sigma_2: A\to C_2$ by $\sigma_2(a):=(1_{B_1}\otimes \rho_2(a))\oplus k(2)N_1\sigma_1(a)$, then we can rewrite the last inequality as $\|\pi_1\sigma_1(a_1)-\sigma_2(a_1)\|<\frac{1}{2}.$ Next, for finite dimensional representations
$$\sigma_2\otimes 1_{k(3)}=k(3)\sigma_2: A\to C_2\otimes\mathbb T_{k(3)}\subseteq \mathbb M_{k(1)k(2)k(3)(N_1+1)}(\mathbb C),$$
and
$$1_{B_2}\otimes\rho_3: A\to B_2\otimes A_3\subseteq \mathbb M_{k(1)k(2)k(3)(N_1+1)}(\mathbb C),$$
choose positive integer $N_2\geq 1$ such that
$$k(3)(N_2+1)\sigma_2: A\to C_2\otimes\mathbb T_{k(3)(N_2+1)}\subseteq \mathbb M_{n(2)}(\mathbb C),$$
and
$$(1_{B_2}\otimes\rho_3)\oplus k(3)N_2\sigma_2: A\to (B_2\otimes A_3)\oplus (C_2\otimes \mathbb T_{k(3)N_2})\subseteq \mathbb M_{n(2)}(\mathbb C),$$
are approximately unitarily equivalent, where $n(2):=k(1)k(2)k(3)(N_1+1)(N_2+1)$. Choose a unitary $u_2\in \mathbb M_{n(2)}(\mathbb C)$ with
$$\|u_2k(3)(N_2+1)\sigma_2(a_j)u_2^*-(1_{B_2}\otimes\rho_3(a_j))\oplus k(3)N_2\sigma_2(a_j)\|<\frac{1}{4},$$
for $j=1,2$. Put $B_3:=u_2(C_2\otimes \mathbb T_{k(3)N_2})u_2^*\subseteq \mathbb M_{n(2)}(\mathbb C)$, and define $\pi_2: C_2\to B_3$ by $\pi_2(c):=u_2(c\otimes 1_{k(3)(N_2+1)})u_2^*$. Put $C_3:=(B_2\otimes A_3)\oplus (C_2\otimes \mathbb T_{k(3)N_2})\subseteq \mathbb M_{n(2)}(\mathbb C)$ and define $\sigma_3: A\to C_3$ by $\sigma_3(a):=(1_{B_2}\otimes \rho_3(a))\oplus k(3)N_2\sigma_2(a)$, then we can rewrite the last inequality as $\|\pi_2\sigma_2(a_j)-\sigma_3(a_j)\|<\frac{1}{4},$ for $j=1,2$.

Since all the algebras involved are unital, we may identify both $B_n$ and $C_n$ with a subalgebra of $C_{n+1}$. Continuing this way, we get finite dimensional triangular algebras $C_n$, injective homomorphisms $\pi_n: C_n\to C_{n+2}$ and u.c.c. maps $\sigma_n: A\to C_n$, satisfying the conditions of the previous lemma (with $k=2$). Therefore, $A$ is TAF-embeddable.
\end{proof}

\begin{remark}
$(i)$ If the condition of approximate unitary equivalence in Proposition \ref{qrfd} is replaced with genuine
unitary equivalence, the statement becomes trivial (at least for the case of C*-algebras, c.f., \cite[Exercice 8.1.1]{bo}).

$(ii)$ Let $\mathfrak F\subseteq A$ be a finite subset and $\varepsilon>0$. A finite dimensional representation $\sigma: A\to \mathbb M_k(\mathbb C)$ is $(\mathfrak F,\varepsilon)$-{\it admissible} if there is a faithful essential representation $\pi: A\to \mathbb B(H)$ on a separable Hilbert space $H$, and a unitary $u: \oplus_1^\infty \ell^2_k\to H$, such that $\|u\pi(a)u^*-\sigma_\infty(a)\|<\varepsilon,\ (a\in\mathfrak F)$, writing $\pi\approx_{(\mathfrak F,\varepsilon)} \sigma_\infty$.  Let us observe that in the proof of Proposition \ref{bdry}, in each step we need to know that the finite dimensional representations involved are homotopic and the one added with large multiplicity is $(\mathfrak F,\varepsilon)$-admissible, for a suitable choice of $(\mathfrak F,\varepsilon)$. Indeed, for two homotopic representations $\sigma_1$ and $\sigma_2$ with $\sigma_1$ being $(\mathfrak F,\varepsilon)$-admissible, there in $N\geq 1$ with $(N+1)\sigma_1\approx_{(\mathfrak F,3\varepsilon)}\sigma_2\oplus\sigma_1$ (this is proved verbatim to \cite[Theorem 8.1.8]{bo}).  
\end{remark}

\begin{remark}
	$(i)$ We warn the reader that the last condition on the range of c.c. homomorphisms $\rho_i$ is a rather strong condition, and an RFD $C^*$-algebra is not necessarily  TRFD as an operator algebra (unless it is commutative). Also, even if $A$ is TRFD and all representations $\rho_i$ are boundary representations, then the enveloping $C^*$-algebra $C^*_e(A)$ may fail to be is RFD (though each  $\rho_i$ uniquely extends to a finite dimensional representation $\tilde\rho_i$ of $C^*_e(A)$, here is no way to guarantee that $\tilde\rho_i$'s are asymptotically isometric).   
	
	$(ii)$ As an example of a RFD operator algebra, let $\{P_n\}$ be an increasing sequence of finite rank projections in $\mathbb B(H)$, SOT-converging to the identiti. Consider,
	$$A:=\{(T_n)\in\prod_{n\geq 1} P_n\mathbb B(H)P_n: \ {\rm there\ exists}\ T\in\mathbb B(H)\ {\rm with}\ T_n\xrightarrow[]{\text{SOT}} T\}.
	$$
	Then $A$ is norm closed: Given $\{(T_n^k)\}_k\subseteq A$, if $(T_n^k)\to (T_n)$ in $\prod_{n\geq 1} P_n\mathbb B(H)P_n$, as $k\to\infty$, choose $T^k\in\mathbb B(H)$ with  $T_n^k\xrightarrow[]{\text{SOT}} T^k$, as $n\to \infty$, then since $\|T_n^k-T_n\|\to 0$, uniformly on $n$, as $k\to\infty$, a standard triangle inequality argument shows that $\{T^k\}$ is SOT-Cauchy, and so there is $T\in\mathbb B(H)$ with $T^k\xrightarrow[]{\text{SOT}} T$. Again using the above uniform convergence in norm, another triangle inequality argument shows that $T_n\xrightarrow[]{\text{SOT}} T$, that is, $(T_n)\in A$. Since SOT is jointly {\it sequentially} continuous  \cite[Page 136]{mur}, $A$ is an operator algebra (but obviously not a C*-algebra). Finally, $A$ is RFD, since for $k(i):={\rm rank}(P_i)$, the sequence of c.c. homomorphisms, 
	$$\rho_i : A\to \mathbb M_{k(i)}(\mathbb C); \ \ (T_n)\mapsto T_i,\ \ (i\geq 1),$$  
	is clearly asymptotically isometric.
	
	$(iii)$ It is not hard to modify the above example such that $A$ is also TRFD: Let $\mathcal T_n\subseteq \mathbb M_{k(n)}(\mathbb C)$ be a (maximal) triangular subalgebra. Let,
	$$A:=\{(T_n)\in\prod_{n\geq 1} \mathcal T_n: \ {\rm there\ exists}\ T\in\mathbb B(\bigoplus_n \ell^2_{k(n)})\ {\rm with}\ T_n\xrightarrow[]{\text{SOT}} T\}.
	$$
	Then $A$ is a TRFD operator algebra. 
\end{remark}

Let $A$ be a unital operator algebra and $\pi: A\to \mathbb B(H)$ be a unital completely contractive map. Let $q: \mathbb B(H)\to \mathbb B(H)/\mathbb K(H)$ be the quotient map onto the Calkin algebra. We say that $\pi$ is a ({\it faithful}) {\it representation modulo the compacts} if $q\circ \pi: A\to  \mathbb B(H)/\mathbb K(H)$ is a (completely isometric) completely contractive homomorphism.

\begin{definition}\label{ec}
	An operator algebra $A\subseteq \mathbb B(H)$ is called {\it essentially closed} if its image $q(A)$ in the Calkin algebra is norm closed. This is equivalent to requiring that  $A+\mathbb K(H)$ to be closed in $\mathbb B(H)$.   
\end{definition}

\begin{example}\label{ex}
	$(i)$  $C^*$-algebras are automatically essentially closed, since the range of $*$-homomorphisms on $C^*$-algebras are closed. More generally, if $B$ is a $C^*$-algebra, $A$ is a $C^*$-subalgebra and $J$ is a closed ideal, then $A+J$ is closed in $B$ \cite[1.8.4]{di}.
	
	$(ii)$  It was shown by Arveson the discrete triangular algebra is essentially closed: every set $\mathcal P$ of projections in $\mathbb B(H)$ determines a weakly closed unital algebra alg$\mathcal P$,
	consisting of all operators $T\in \mathbb B(H)$ satisfying $(1 - P) TP = 0$, for every $P\in\mathcal P$. Given an increasing
	sequence $\{P_n\}$ of finite rank  projections such that $P_n\uparrow	1$ in SOT, $\mathcal T:=$alg$\{P_n\}$ is essentially closed \cite[Proposition 2.1]{ar}. Indeed, $\mathcal T+\mathbb K(H)=\mathcal{QT}$, consisting of all quasitriangular operators w.r.t. $\{P_n\}$ \cite[Corollary of Theorem 2.2]{ar}.
	
		$(ii)$  The result of Arveson is extended by  Loebl and  Muhly  to all nest algebras \cite{lm}. 
		
		$(iii)$ The algebra
		of subnormal operators and Toeplitz algebra are also known to be essentially closed \cite{df}.
		
		$(iv)$ There is a commutative weakly closed operator algebra
		$A$ such that $A+\mathbb K(H)$ is not
		closed \cite[Proposition 2]{df}: let $\{e_n\}$ be a basis for $\ell^2(\mathbb Z)$. Let
		$U$ be the bilateral shift $P$ be the rank one projection onto $\mathbb C e_0$. Let $T = U + P$ and let
		$A$ be the weakly closed algebra generated by $T$, then $A\cap \mathbb K(H)=0$ and $\|T^n\|\geq\sqrt{n}$ while $\|q(T^n)\|=\|q(U^n)\|=1$, and so $A+\mathbb K(H)$ is not closed by Remark \ref{rk}$(i)$. 
\end{example}

\begin{remark}\label{rk}
	$(i)$ An operator algebra $A\subseteq \mathbb B(H)$ is essentially closed if and only if $A/(A\cap \mathbb K(H))$ and $q(A) =(A+\mathbb K(H))/\mathbb K(H)$ are isomorphic as Banach algebras \cite[Proposition 1]{df}. When $A\cap \mathbb K(H)=0$, this is equivalent to the quotient map $q$ being bounded below on $A$ \cite[Corollary 1]{df}.
	
	$(ii)$ The assumption than an operator algebra $A\subseteq \mathbb B(H)$ is essentially closed doesn't mean that the canonical map from $A/(A\cap \mathbb K(H))$ to $(A+\mathbb K(H))/\mathbb K(H)$ is an isometry (where as the conserve is clearly true,  since complete subspaces of Banach spaces are closed). 
	
	$(iii)$ As a special case of Example \ref{ex}$(i)$,  if $I$ and $J$ are closed two-sided ideals of a $C^*$-algebra, then $I + J$.  Combes and Perdrizet extended this observation by showing that the sum of a closed left ideal and a closed right ideal of a $C^*$-algebra is always closed \cite[Proposition 6.2]{cr}. This
	was rediscovered by Rudin \cite[Example 4.6]{ru} (see also, \cite[Corollary 32.2]{dw}), and again
	by Kirchberg \cite[Lemma 4.9(iv)]{k}, and Wassermann \cite[Lemma 8.1]{wa}. In fact, more is true:
	if $L$ is a closed left ideal and $R$ is a closed right ideal in a
	Banach algebra $B$ such that $L$ has a bounded right approximate identity, or
	$R$ has a bounded left approximate identity, then $L + R$ is closed. If moreover, the approximate identity is bounded by 1, then the canonical map from $L/(L\cap R)$ to $(L+R)/R$ is an isometry  \cite[Proposition 2.4]{dix}.

	$(iv)$ If $B$ is an approximately unital
	operator algebra (i.e., an operator algebra with a contractive approximate identity) and $I$ and $J$ are ideals in $B$ such that $I$ has a contractive approximate identity (cai), then 
	$I/(I \cap J) = (I + J)/J$,  completely isometrically isomorphically. In particular, $I + J$ is closed in $B$ \cite[Theorem 2.3]{abr} (the closedness of $I + J$ also follows under the weaker condition that $I$ is only a hereditary subalgebra with cai, where hereditary here means that $IBI\subseteq I$ \cite[Corollary 4.3]{abr}). A closed inspection of the proof shows that a complete isometric version of part $(ii)$ above also holds here: if $B$ is an approximately unital
	operator algebra and $I$ and $J$ are left and right ideals in $B$, respectively,  such that $I$ has a contractive right approximate identity, then 
	$I/(I \cap J) = (I + J)/J$,  completely isometrically isomorphically. In particular, for $B=\mathbb B(H)$ and $J=\mathbb K(H)$, if an operator algebra $A\subseteq \mathbb B(H)$ is a left ideal in $\mathbb B(H)$ with a contractive right approximate identity, then  the canonical map from $A/(A\cap \mathbb K(H))$ to $(A+\mathbb K(H))/\mathbb K(H)$ is not only an isometry, but indeed a complete isometry. In particular, this holds when $A$ is unital and a left ideal (indeed following the proof, one can see that it is enough to have $\mathbb K(H)A\subseteq A$), however such a condition might be too strong (for instance when we further assume that $A\cap \mathbb K(H)=0$, then $\mathbb K(H)A\subseteq A$ could not hold unless $A=0$).
	
		$(v)$ It is not true that if the restriction of the quotient map $q:\mathbb B(H)\to \mathbb B(H)/\mathbb K(H)$ to an operator algebra $A \subseteq \mathbb B(H)$  is a complete isometry, then the same holds for the  C*-algebra $C^*(A)$ generated by $A$: let $A$ be the norm-closed algebra generated by the unilateral shift $S$ in $\mathbb B(\ell^2(\mathbb Z))$.
	This is complete isometrically isomorphic to the disc algebra, and the image in the Calkin algebra is the same, while $C^*(S)\supseteq\mathbb K(\ell^2(\mathbb Z))$.  
\end{remark}

Remark \ref{rk}$(iv)$ motivates the following stronger version of Definition \ref{ec}.

\begin{definition}\label{ee}
	An operator algebra $A\subseteq \mathbb B(H)$ is called {\it essentially embedded} if the canonical map from $A/(A\cap \mathbb K(H))$ to $(A+\mathbb K(H))/\mathbb K(H)$ is  a complete isometry. 
\end{definition}

\begin{example}
	$(i)$ All $C^*$-algebras are automatically essentially embedded, since injective $*$-homomorphisms on $C^*$-algebras are automatically complete isometric. 
	
	$(ii)$ The discrete  triangular algebra $\mathcal T$ considered by Arveson is  essentially embedded \cite[Proposition 2.1]{ar}: for $T\in \mathbb B(H)$,
	$\|q(T)\|=\lim_n\|(1 - P_n) T(1 - P_n)\|$ by \cite[Lemma 1, page 292]{ar11}. Indeed,  the limit on the RHS is clearly 0 when $T$ is of rank 1, and so when $T$ is of finite rank, and so it also happens when $T$ is compact, as the set of operators for which RHS=0 is closed under norm.
	Next $\limsup_n \|(1 - P_n) T(1 - P_n)\|=\limsup_n \|(1 - P_n) (T+K)(1 - P_i)\|\leq \|T+K\|$, for each $K\in \mathbb K(H)$ and so  $\limsup_i \|(1 - P_n) T(1 - P_n)\|\leq \|q(T)\|$, where as, $\|q(T)\|\leq \liminf_n \|(1 - P_n) T(1 - P_n)\|$, as each $(1 - P_n) T(1 - P_n)$ is a finite rank perturbation of $T$. Now if  $T\in \mathcal T$ then 
	$(1 - P_n)T(1 - P_n)=T + K_n$, where $K_n =
	-P_nT - TP_n + P_nTP_i$ is a finite rank operator in $\mathcal T$ (this is a crucial part of the proof, as in general one cannot guarantee that $K_n$ is in the given operator algebra). This implies
	that $\|q(T)\|\geq\|T+(A\cap \mathbb K(H))\|$, while the inequality in other direction always hold. A similar argument for amplifications shows that the quotient map $q$ is indeed a complete isometry.
	
	$(iii)$  If $S$ is a weak$^*$-closed subspace of $\mathbb B(H)$ such that 
	$S\cap \mathbb K(H)$ is weak$^*$-dense in $S$, then $S$ is essentially embedded \cite[ Corollary 11.7]{da1}.    
	
	$(iv)$ Any non-commutative irreducible algebra $A$ of almost normal
	operators containing the identity fails to be essentially embedded (since the identity representation on $C^*(A)$ is a boundary
	representation \cite[Theorem 2.2.2]{ar11}, see Remark \ref{bdry}$(iii)$).
	
\end{example}
The next lemma related the notion of essential embedding to representations modulo the compacts.

\begin{lemma}\label{mod}
	Let $A$ be an  operator algebra and the complete contraction  $\pi: A\to \mathbb B(H)$  be a faithful representation modulo the compacts, then the operator algebra generated by $\pi(A)$ is essentially embedded in $\mathbb B(H)$.
\end{lemma}
\begin{proof}
Let $q$ be the quotient map onto the Calkin algebra on $H$. Since $\pi$  is faithful  modulo the compacts, for each  $a\in A$, 
$$\|q(\pi(a))\|=\|a\|\geq \|\pi(a)\|,$$
and as the reverse inequality is automatic, $q$ is isometric on $\pi(A)$. Since $q$ is a continuous homomorphism, it follows that $q$ is also isometric on the operator algebra generated by $\pi(A)$. A similar argument on amplifications shows that $q$ is indeed a complete isometry on this operator algebra.
\end{proof}

\begin{remark}
$(i)$ Note that the converse of the above lemma is not true: if the operator algebra generated by the range of a complete contraction  $\pi: A\to \mathbb B(H)$ is essentially embedded, $\pi$ may fail to be  faithful modulo the compacts (indeed, if $\pi$ is not an isometry, then so is $q\circ\pi$). 

$(ii)$ One could use the above lemma to construct examples of operator algebras which are essentially embedded and QT. 

$(iii)$ More examples could be constructed using nest algebras as follows: Let $\mathcal P:=\{P_nH\}$ be the nest consisting of the ranges of finite rank projection SOT-increasing to the identity. Let $\mathcal T(\mathcal P)$ and $\mathcal {QT}(\mathcal P)$ be the set of operators $T$ with $(1-P_n)TP_n=0$ and $(1-P_n)TP_n\to 0$, respectively. It follows from \cite[Theorem 
12.2]{da1} that $\mathcal {QT}(\mathcal P)=\mathcal {T}(\mathcal P)+\mathbb K(H)$ (we warn the reader that $\mathcal {T}(\mathcal N)+\mathbb K(H)$ need not consist of  
quasitriangular operators, for a general nest $\mathcal N$; c.f., \cite[Chapter 21]{da1}). 
By \cite[Theorem 12.1]{da1}, $\mathcal {QT}(\mathcal P)$ is norm closed, and 
the quotient map:  $$\sigma: \mathcal {T}(\mathcal P)/(\mathcal {T}(\mathcal P)\cap \mathbb K(H))\to (\mathcal {T}(\mathcal P)+\mathbb K(H))/\mathbb K(H)$$ is isometric (Indeed, $\mathcal {T}(\mathcal P)$ is weak$^*$-closed and by the Erd\"os Density Theorem \cite[Theorem 3.11]{da1}, $\mathcal {T}(\mathcal P)\cap \mathbb K(H)$ is weak$^*$ dense, and so an M-ideal in $\mathcal {T}(\mathcal P)$ \cite[Theorem 11.5, Corollary 11.6]{da1} and so the quotient map is isometric by \cite[Theorem 11.3]{da1}). Repeating this argument for the inflation $\sigma^{(n)}$ for the Hilbert space $H\otimes \ell^2_n$, we get that $\sigma$ is indeed completely isometric). Thus $\mathcal {T}(\mathcal P)$ is essentially embedded. Since $\mathcal {QT}(\mathcal P)+\mathbb K(H)=\mathcal {T}(\mathcal P)+\mathbb K(H)$ and $\mathcal {QT}(\mathcal P)\cap\mathbb K(H)=\mathcal {T}(\mathcal P)\cap\mathbb K(H)$, it follows that $\mathcal {QT}(\mathcal P)$ is also essentially embedded. Now, $\mathcal {QT}(\mathcal P)$ is clearly QT, but it is not triangular or block triangular, as it contains all compact operators.     
\end{remark}

\begin{lemma}\label{lemma4}
	Let $A\subseteq \mathbb B(H)$ be an essentially embedded operator algebra with $A\cap \mathbb K(H)=0$. Then each  c.c. representation $\pi: A\to \mathbb B(K)$  extends to c.c. representation  $\bar\pi: A+\mathbb K(H)\to \mathbb B(K)$ which is zero on $\mathbb K(H)$.   
\end{lemma}
\begin{proof}
	Since $A$ is essentially embedded and $A\cap \mathbb K(H)=0$, there is a surjective c.i. homomorphism $\sigma: (A+\mathbb K(H))/\mathbb K(H)\to A$. The composition $\bar\pi:=\pi\circ\sigma\circ q: A+\mathbb K(H)\to \mathbb B(K)$ has all the required properties.   
\end{proof}

 The next result is an slight extension of Voiculescu-Weyl-von Neumann theorem to separable operator algebras.

\begin{lemma}\label{lemma2}
	Let $A\subseteq \mathbb B(H)$ be a separable unital essentially embedded operator algebra  and $\pi: A\to \mathbb B(K)$ be a u.c.c. map with $\pi(A\cap \mathbb K(H))=0$. Then there are isometries $V_n: K\to H$ with $\pi(a)-V_n^*aV_n\in \mathbb K(K)$, for each $a\in A$ and $n\geq 1$, such that ad${}_{V_n}\to \pi$, point-norm on $A$.
\end{lemma}
\begin{proof}
By Lemma \ref{lemma4}, we get a c.c. representation  $\bar\pi: A+\mathbb K(H)\to \mathbb B(K)$ which vanishes on $\mathbb K(H)$. By \cite[1.2.3, 1.2.8]{ar1}, $\bar\pi$ extends to a completely positive map $\tilde\pi: \mathbb B(H)\to \mathbb B(K)$, still vanishing on $\mathbb K(H)$. Thus without loss of generality we may assume that $A$ contains $\mathbb K(H)$ and $\pi$ vanishes on $\mathbb K(H)$. Next, the $C^*$-algebra $B$ generated by the operator system $\overline{A+A^*}$ in $\mathbb B(H)$ is a separable $C^*$-algebra containing $\mathbb K(H)$ satisfying $\tilde\pi(B\cap \mathbb K(H))=\tilde\pi(\mathbb K(H))=0$. 
By the Voiculescu-Weyl-von Neumann theorem (c.f. \cite[II.5.3]{da}), there are isometries $V_n: K\to H$ with $\tilde\pi(b)-V_n^*bV_n\in \mathbb K(K)$, for each $b\in B$ and $n\geq 1$ such that ad${}_{V_n}\to \tilde\pi$, point-norm on $B$. The same isometries do the job for $A$ as well.
\end{proof}

\begin{lemma}\label{qt2}
	Let $A\subseteq \mathbb B(H)$ be a separable  operator algebra  and let $(\pi, H_\pi)$ and $(\sigma, H_\sigma)$ be representations of $A$ such that there are partial isometries $V_n: H_\pi\to H_\sigma$ with image projection 1 such that $\|\sigma(a)V_n-V_n\pi(a)\|\to 0$, for $a\in A$. Then if $\sigma(A)$ is a QT set of operators, so is $\pi(A)$.
\end{lemma}
\begin{proof}
	By definition, for finite subsets $F_n\subseteq H_\sigma$ and $\mathfrak F_n\subseteq A$, there are finite rank projections $P_n\in B(H_\sigma)$ such that
	$$\|P_n\sigma(a)-P_n\sigma(a)P_n\|\approx
0, \ \|P_n\xi\|\approx\|\xi\|\ \ \ (a\in \mathfrak F_n, \xi\in F_n).$$ 
Given $G_n\subseteq H_\pi$, let $F_n:=V_n(G_n)$ and $Q_n:=V_n^*P_nV_n$. Then each $Q_n\in B(H_\pi)$ is a projection, since $V_nV_n^*=1$, and it is of finite rank, since $P_n$ is so. Also,
\begin{align*}
	\|Q_n\pi(a)-Q_n\pi(a)Q_n\| &=\|V_n^*P_nV_n\pi(a)-V_n^*P_nV_n\pi(a)V_n^*P_nV_n\|\\
	&=\|V_n^*P_n\sigma(a)V_n-V_n^*P_n\sigma(a)P_nV_n\|\\
	&\leq\|P_n\sigma(a)-P_n\sigma(a)P_n\|\approx 0,
\end{align*} 
as required.
\end{proof}

Now we are ready to prove the analog of Voiculescu theorem for QT algebras. The original proof (c.f. \cite[7.2.5]{bo}) is based on Voiculescu constants 
$$\eta_{\pi}(a):=2{\rm max}\big(\|\pi(a^*a)-\pi(a^*)\pi(a)\|^{\frac{1}{2}}, \|\pi(aa^*)-\pi(a)\pi(a^*)\|^{\frac{1}{2}}\big),$$
defined for faithful representations $\pi $ modulo the compacts for the case  of $C^*$-algebras (c.f. \cite[1.7.6]{bo}). This argument is obviously not usable for operator algebras, and instead we employ the above lemma (which gives an alternative way, also suggested by the work of Voiculescu). Also note that a proof by reduction to the  $C^*$-algebra generated by the given operator algebra does not seem to be possible, as it is not the case that if an operator algebra is QT as a family of operators, the it generated a  $C^*$-algebra which is a QD family of operators (this is because the notion of QT families defined by Halmos is an upper quasitriangular family which switches to a lower quasitriangular family, when taking adjoint).

\begin{proof}[Proof of Theorem~\ref{qtr}] $(i)\Rightarrow (iii).$ Since $A$ is separable, we may choose a {\it sequence} of completely contractive maps $\phi_n: A\to \mathbb M_{k_n}(\mathbb C)$ which are asymptotically isometric and asymptotically multiplicative, that is, 
$$\|\phi_n(a)\|\approx \|a\|, \ \ \|\phi_n(ab)-\phi_n(a)\phi_n(a)\|\approx 0,$$
for $a\in F_n$, where $1\in F_1\subseteq F_2\subseteq \cdots$ is an increasing sequence of finite subsets of $A$ with dense union. Take any faithful unital essential representation $\pi: A\to \mathbb B(H)$ in a separable Hilbert space $H$. Consider the faithful representation modulo the compacts $\Phi:=\oplus\phi_n: A\to \prod_{n=1}^{\infty} \mathbb M_{k_n}(\mathbb C)$, which is a u.c.c. map. Compose $\Phi$ with the quotient map $q$ from $B(\bigoplus_{n=1}^{\infty}\ell^2_{k(n)})$ onto its Calkin algebra and embed the latter in $\mathbb B(K)$, for some Hilbert space $K$, to get a u.c.c. map $\Pi:=q\circ \Phi: A\to \mathbb B(K)$. Since $\pi$ is essential,  we may assume that $A\subseteq \mathbb B(H)$ with $\pi(A\cap \mathbb K(H))=0$. By Lemma \ref{lemma2} (switching the role of $V_n$ and $V_n^*$), there are partial isometries $V_n: H\to K$ with range projection 1 such that $\Pi(a)-V_n\pi(a)V_n^*\in \mathbb K(K)$, for each $a\in A$ and $n\geq 1$, and
$$\|\Pi(a)-V_n\pi(a)V_n^*\|\approx 0, \ \ (a\in A).$$

 By Lemma \ref{qt2}, it is enough to show that $\Pi(A)$ is a QT set of operators, which is in turn follows if we show that $\Phi(A)$ is a QT set of operators (since $q$ is norm decreasing). Put $X:=\bigoplus_{n=1}^{\infty}\ell^2_{k(n)}$ and let $X^{(k)}$ and $X^{(\infty)}$ be the direct sum of $k$ and countably many copies of $X$ with itself, respectively. Consider representations $\Phi^{(k)}$ and $\Phi^{(\infty)}$ of $A$ on  $X^{(k)}$ and $X^{(\infty)}$. Let $P_n$ be the projection in $X$ onto $\ell^2_{k(n)}$, and not that the relations of asymptotic isometry and asymptotic multiplicativity could be rewritten in terms of projections $P_n$ as
  $$\|P_n\Phi(a)P_n\|\approx \|a\|, \ \ \|P_n\Phi(a)-P_n\Phi(a)P_n\|\approx 0,$$
  for $a\in F_n$. Let $X_k:=\bigoplus_{n=k}^{\infty}\ell^2_{k(n)}$and let $Q_k$ be the orthogonal projection onto $X_k$. Then
  $$[Q_k, \Phi(a)]=\bigoplus_{n\geq k} [P_n, \phi_n(a)]\in \mathbb K(X),$$
  for $a\in\cup_{n} F_n$, and so by density for any $a\in A$. Now since ${\rm ad}_{Q_n}\circ \Phi: A\to B(X)$ is a faithful representation modulo the compacts and $A$ is essentially embedded, then 
  applying Lemma \ref{lemma2} again, there are partial isometries $U_n: X^{(k)}\to X$ with $\Im(V_n)\subseteq X_n$ such that $\|U_n\Phi^{(k)}(a)-\Phi(a)U_n\|\approx 0$, for $a\in A$. To see that $\Phi(A)$ is a QT set of operators, without loss of generality we may consider the finite subsets $\Phi(F_n)\subseteq \Phi(A)$ and $F\subseteq X^{(k)}$ and first choose $N_n\geq 1$ such that for the orthogonal projection $\tilde Q_n$ onto $\bigoplus_{i=n}^{N-n}\ell^2_{k(i)}$,
  $$\|\tilde Q_n\Phi(a)-\tilde Q_n\Phi(a)\tilde Q_n\|\approx 0,\ \ \|(1-\tilde Q_n)U_n\xi\|\approx 0\ \ (a\in F_n, \xi\in F).$$  
 Next, let $W_k: X\to X^{(\infty)}$ be the isometry mapping $X$ onto its $(k+1)$-copy in $X^{(\infty)}$ and put $\tilde W_n:=U_n^*\oplus W_n(1-U_nU_n^*)$, and observe that 
$\|\tilde W_n\Phi(a)-\Phi^{(\infty)}(a)\tilde W_n\|\approx 0$, for $a\in A$. Finally, put $\tilde P_n:=\tilde W_n\tilde Q_n\tilde W_n^*$, and observe that 
$$\|\tilde P_n\Phi(a)-\tilde P_n\Phi(a)\tilde P_n\|\leq \|\tilde Q_n\Phi(a)-\tilde Q_n\Phi(a)\tilde Q_n\|\approx 0,$$
and 
$$\|\tilde P_n\xi\|\leq \|\tilde Q_nU_n\xi\|\approx \|\xi\|,$$
for $a\in F_n$ and $\xi\in F$, as required. 
 
 $(iii)\Rightarrow (ii).$ This is immediate.
 
 $(ii)\Rightarrow (i).$ If $\pi: A\to \mathbb B(H)$ is a faithful QT  representation on a separable Hilbert space and $P_n$'s are finite rank projections with $P_n\uparrow 1$ in SOT, and $\|P_na-P_naP_n\|\to 0$, for each $a\in A$, as $n\rightarrow\infty$. If $k(n)$ is the rank of $P_n$, we identify $P_n\mathbb B(H)P_n$ with $\mathbb M_{k(n)}(\mathbb C)$ and regard $\phi_n(a):=P_naP_n$ as a map from $A$ into $\mathbb M_{k(n)}(\mathbb C)$. Then
 \begin{align*}
\|\phi_n(ab)-\phi_n(a)\phi_n(b)\|&=\|P_nabP_n-P_naP_nbP_n\|\\
 &\leq\|P_nab-P_naP_nb\|\\
  &\leq\|P_na-P_naP_n\|\|b\|\approx 0,
 \end{align*}
 for each $a,b\in A$. Also,
$$ \|\phi_n(a)-a\|=\|P_naP_n-a\|\approx \|P_naP_n-P_na\|\approx 0,$$
  for each $a\in A$. 
\end{proof}

\begin{remark}\label{bdry}
	$(i)$ A {\it boundary representation} of a unital operator algebra $A$ consists
	of a completely isometric homomorphism $\phi: A \to C$, where $C$ is a C*-algebra
	and $C^*(\phi(A)) = C$, together with a representation $\pi: C \to\mathbb B(H)$ such that
	the only completely positive map on $C$ agreeing with $\pi$ on $\phi(A)$ is $\pi$ itself. (In
	the original definition by Arveson, boundary representations are also assumed to be
	irreducible, but this condition was dropped later). It immediately follows that in Theorem~\ref{qtr}, in part $(iii)$ ``essential  representation'' could be replaced by ``essential boundary representation''. Furthermore, since every separable operator algebras is known to posses a faithful boundary representation \cite{ar2}, \cite{k} the same replacement could be done in part $(ii)$. This observation is  for those who prefer to work with homomorphisms with extension to $*$-homomorphisms on the $C^*$-algebra generated by the image of $A$. The boundary representations appear in the study of the noncommutative Choquet boundary (the peak points of the Shilov boundary) of operator systems \cite{ar2}. Arveson showed that unital operator algebras have enough boundary representations to generate the enveloping $C^*$-algebras. This is extended to operator systems by Ken Davidson and Matt Kennedy \cite{dk}, and reconfirmed by a result of Dritschel and McCullough who showed (dropping the irreducibility condition of Arveson) that any family of representations of an operator algebra or an operator space (in Agler’s sense) has boundary representations \cite{dm}, and thereby gave a direct proof of Arveson's result (compare to the results in \cite{h}).
	
	$(ii)$ One might suspect that since boundary representations extend to  representations of  an enveloping $C^*$-algebra, an indirect proof of the above result is possible for such representations by extending and using Voiculescu theorem for the enveloping $C^*$-algebra, but this is not possible, since the the enveloping $C^*$-algebra of a QT operator algebra is not necessarily QD. Also, the extension of a faithful (resp., essential) boundary representation need not be faithful (resp., essential). 
	
	$(iii)$ The boundary theorem of Arveson relates boundary representations to essentially embedded (non closed) operator systems as follows:  if  $S$ is an irreducible set of operators on a Hilbert space $H$ containing the identity and the enveloping $C^*$-algebra $C_e^*(S)$ contains $\mathbb K(H)$. Then the identity representation of $C_e^*(S)$ is a boundary representation
	on $S$ if and only if, the quotient map $q: \mathbb B(H)\to \mathbb B(H)/\mathbb K(H)$ is not completely isometric on the
	linear span of $S\cup S^*$ \cite[Theorem 2.1.1]{ar11}.   
	\end{remark}

	To prove Proposition \ref{fg} we need a bit of preparation. The proof of the next lemma is an easy adaptation of that of \cite[Lemma 15.3]{b}, which is given here just for the sake of completeness. 
	
\begin{lemma}\label{lemma3}
Let $A$ be a separable operator algebra with a faithful representation with QT range (as a set of operators), then $A$ has such a representation on a separable Hilbert space.
\end{lemma}
\begin{proof}
We shall prove a bit more by showing that for each faithful representation $\pi : A \to \mathbb B(H)$ with QT range, $H$ has a $\pi$-invariant separable subspace $K$ such that the corresponding sub-representation $\pi^K$ has a QT range. We do this by constructing an increasing sequence $(K_n)$ of separable $\pi$-invariant
subspaces of $H$ and finite rank projections $Q_n$ with range inside $K_{n+1}$ such that $\|Q_n\pi(a)-Q_n\pi(a)Q_n\|\to 0$, as $n\to\infty$ and $Q_n\to 1$ (SOT) on the union of $K_n$'s.   

By separability of $A$, we may choose a dense sequence $(a_i)$ in the unit ball
of $A$ and unit vectors $\xi_{i,n}$ with $\|\pi(a_i)\xi_{i,n}\|>\|a_i\|-2^{-n}$, for each $i,n$. Let $H_1$ be the span closure of the set of all vectors $\xi_{i,n}$ and $K_1$ be the closure of $\pi(A)H_1$. This is a separable $\pi$-invariant subspace of $H$ and the sub-representation $\pi^{K_1}$, being isometric on the sequence $(a_i)$, is faithful. Let $(e_{1,k})$ be an ONB for $K_1$, then by assumption there is a finite rank projection $Q_1$ in $\mathbb B(H)$ with $\|Q_1\pi(a_1)-Q_1\pi(a_1)Q_1\|< 2^{-1}$ and $\|Q_1(e_{1,1})-e_{1,1}\|$ as small as we wish. Indeed by an argument similar to that of \cite[Proposition 3.4]{b}, we may also arrange for equality $Q_1(e_{1,1})=e_{1,1}$ to hold. Let's do so and take $H_2$ the be the span closure of Im$(Q_1)\cup K_1$ and $K_2$ be the closure of $\pi(A)H_2$. Let $(e_{2,n})$ be an ONB for $K_2$ and choose a finite rank projection $Q_2$ in $\mathbb B(H)$ with $\|Q_2\pi(a_i)-Q_2\pi(a_i)Q_2\|< 2^{-2}$, for $i=1,2$, such that $Q_2$ fixes a finite ONB of Im$(Q_1)$ (which means that $Q_2\geq Q_1$). Proceeding inductively, we get an
increasing sequence of separable $\pi$-invariant subspaces $K_n$ with ONB $(e_{n,k})$
and increasing sequence of finite rank projections $Q_n$ with range inside $K_{n+1}$ satisfying $Q_n(e_{i,j})=e_{i,j}$, for $1\leq i,j\leq n$, and $\|Q_n\pi(a_i)-Q_n\pi(a_i)Q_n\|< 2^{-n}$, for $1\leq i\leq n$. Now the closure $K$ of the union of $K_n$'s is a separable $\pi$-invariant subspace of $H$ and the range of the corresponding (faithful) sub-representation $\pi^{K}$ is QT.
\end{proof}

Now let $A$ be a separable operator algebra and $\pi : A \to \mathbb B(H)$ be a faithful representation. Then by the above lemma and an argument as in the proof of \cite[Lemma 15.4]{b} there exists a separable $\pi$-invariant subspace $K$ of $H$ such that $\pi^K$ is faithful with $\pi^K(a)$ finite rank if and only if $\pi(a)$ is so (with equal ranks). This could be used to make another useful observation in the spirit of the previous lemma: for each faithful representation $\pi: A\to \mathbb B(H)$, if $L$ is the span closure of $\pi(A)H$, then $\pi$ has a QT range (as a set of operators) iff its non-degenerate sub-representation $\pi^L$ is so (it is wise to handle the unital case first, c.f. \cite[Lemma 15.5]{b}). 

Next take any faithful essential representation $\pi: A\to \mathbb B(H)$. Since quasitriangularity of a set of operators is defined via approximation on finite sets, the range of $\pi$ is a QT set of operators if $\pi(B)$ is QT for each separable operator subalgebra $B$ of $A$. Moreover, when $H$ is non separable, for each  finite subset  $F\subseteq H$, by the first observation in the last paragraph, there is a  separable $\pi$-invariant subspace $K$ of $H$ containing $F$ such that the corresponding sub--representation $\pi^K$ is faithful and essential. This plus Theorem~\ref{qtr} proves the next lemma, which shows that part of Theorem~\ref{qtr}  also holds for the non separable case. This is recorded in the next lemma. 

\begin{lemma}\label{lemma5}
Let $A$ be a QT  operator algebra then the range of each faithful essential representation of $A$ is QT set of
operators.
\end{lemma}

\begin{proof}[Proof of Proposition \ref{fg}]  The necessity is obvious. To prove the sufficiency, let $\pi : A \to \mathbb B(H)$ be a faithful essential
representation. Then for each finitely generated subalgebra $B$ of $A$, the restriction of $\pi$ to $B$ is a faithful essential representation and so by Lemma \ref{lemma5}, the image of $B$ under $\pi$ is a QT set of operators. Since this holds for each finitely generated subalgebra, it immediately follows from the definition that $\pi(A)$ is also QT. 
\end{proof} 

In the last stage of the above proof, note that although $B$ is separable, we still need to appeal to the above lemma as $H$ is not assumed to be separable.   

The next proof is an adaptation of an argument by Arveson \cite[Theorem 2]{ar3}. 

\begin{proof}[Proof of Theorem~\ref{rfd}] We only need to show the necessity,  since $B + \mathbb K(H)$ is a quasitriangular set of operators. We may assume that  $F$ is contained in
the unit ball of $A$. Take an increasing sequence $F\subseteq F_1 \subseteq F_2 \subseteq F_3 \subseteq\dots$ of finite sets with dense union in the unit ball of $A$. Again by  Theorem \ref{qtr}, $A$ is a quasitriangular set of operators and there is an increasing sequence of  finite rank projections $P_n$ converging to 1 (SOT)  such that $\|P_na- P_naP_n\| \to 0$ for each $a\in A$. Passing to a subsequence, we let $\|P_na- P_naP_n\| < \varepsilon/2n,$ for 
$a \in F_n$. Put $P_0 = 0$ and let $E_n = P_n - P_{n-1}$, then $\sum_n E_n = 1$ in SOT and $a\mapsto \delta(a):=\sum_n E_naE_n$ is a complete contraction on $A$ satisfying   
$a-\delta(a)=\sum_n (aE_n - E_naE_n)$ in SOT. Let $B$ be the operator algebra generated by the image of $\delta$ in $\mathbb B(H)$, which is clearly a block triangular algebra. When $a\in F_n$ for some $n\geq 1$, the last sum also converges in
the norm and $a-\delta(a)\in \mathbb K(H)$. By density of the union of $F_n$'s and norm continuity of $\delta$,  $a-\delta(a)\in \mathbb K(H)$ for each $a \in  A$. This in particular shows that $A$ is essentially closed iff $B$ is so. Since $a-\delta(a)\in \mathbb K(H)$, $\delta$ induces a well-defined linear map 
$$\tilde\delta: (A+\mathbb K(H))/\mathbb K(H)\to (B+\mathbb K(H))/\mathbb K(H);\ \ a+\mathbb K(H)\mapsto \delta(a)+\mathbb K(H),$$
which is contractive, since 
$$\|a+k\|=\|\delta(a)+(k+a-\delta(a))\|\geq \|\delta(a)+\mathbb K(H)\|,$$
for each $k\in \mathbb K(H)$, thus taking infimum over $k$, $\|a+\mathbb K(H)\|\geq \|\delta(a)+\mathbb K(H)\|$. Indeed, $\tilde\delta$ is an isometry, as the above argument could be symmetrically repeated the other way around. Now, by the same argument applied to amplifications of $\delta$, $\tilde\delta$ is also a complete isometry. Also, $\delta$ is a homomorphism modulo compacts, since
\begin{align*}
	\delta(ab)-\delta(a)\delta(b)&=(\delta(ab)-ab)-(\delta(a)-a)\delta(b)-\delta(a)(\delta(b)-b)\\&+(\delta(a)-a)(\delta(b)-b)\in\mathbb K(H),
\end{align*}
for each $a,b\in A$, therefore, $\tilde\delta$ is a homomorphism. 

Now if $A$ is essentially embedded and meets compacts trivially, the restriction of the quotient map $q:\mathbb B(H)\to\mathbb B(H)/\mathbb K(H)$ to $A$ is a complete isometry. On the other hand, since $B$ is block-triangular, it is essentially embedded by Example \ref{ee}$(ii)$, therefore, the canonical map $q_B: B/(B\cap \mathbb K(H))\to (B+\mathbb K(H))/\mathbb K(H)$ is a complete isometry, and so is the composition map $q_B^{-1}\circ \tilde\delta\circ (q|_A): A\to B/(B\cap \mathbb K(H))$.   
\end{proof}

Two operator algebras $A$ and $B$ are {\it homotopic} if there are homomorphisms $\phi: A\to B$ and $\psi: B\to A$ such that $\phi\circ\psi$ is homotopic to the identity on $B$ and $\psi\circ\phi$ is homotopic to the identity on $A$. It is desirable to show that quasitriangularity passes from one operator algebra to the other when they are homotopic, but at this point it is not clear if the  argument  in \cite[Theorem 5]{v} passes to non selfadjoint operator algebras. In particular, it is not known to us if {\it contractible} operator algebras--i.e., those homotopic to the zero algebra--are quasitriangular.

Fortunately, this could be taken care of for cones. The notions of the {\it cone} and {\it suspension} could be defined for an operator algebra $A$ by $CA:=C_0(0,1]\otimes A$ and since $C_0(0,1]$ is a nuclear $C^*$-algebra, the min and max tensor products are the same here \cite[Proposition 2.9]{pp}. The same holds for the suspension $SA:=C_0(0,1)\otimes A$. If an operator algebra $A$ is sitting in a C*-algebra $D$, then  $CA$ sits in the QD C*-algebra $CD$, and so $CA$ is QT. 

\begin{remark}
	Indeed more is true by a result of Salinas \cite{s}. In order to describe the result of Salinas we need some notations: let $D$ be a unital separable C*-algebra and let $Ext(D)$ be the set of equivalence classes of unital $*$-monomorphisms of $D$ into the Calkin algebra $Q$, and $Ext_{-1}(D)$ be the group of invertible in $Ext(D)$. An extension $\tau$ is said to be QT w.r.t. a closed subalgebra $A$ if $\pi^{-1}(\tau(A))$ is QT, where $\pi$ is the quotient map onto the Calkin algebra. Salinas showed that if $B$ sits in a unital separable C*-algebra $A$ and $\gamma: [0,1]\to Ext_{-1}(D)$ is continuous so that $\gamma(0)$ is QT then so is $\gamma(1)$ \cite[Theorem 5.10]{s}.
\end{remark}

\begin{remark} \label{essentail}
	$(i)$ If the embedding $A\subseteq \mathbb B(H)$ is not essential, then $\tilde A:=A/(A\cap \mathbb K(H))$ has a natural essentail embedding: consider the chain of identification and inclusions
	$$\tilde A= (A+\mathbb K(H))/\mathbb K(H)\subseteq \mathbb B(H)/\mathbb K(H)\subseteq \mathbb B(K),$$
	for a Hilbert space $K$. It is known that the Calkin algebra $\mathbb B(H)/\mathbb K(H)$ is antiliminal \cite[4.7.22(b)]{dix}, so it has no non-zero liminal closed ideal  \cite[page 99]{dix}, in particular,  $(\mathbb B(H)/\mathbb K(H))\cap \mathbb K(K)=0$, thus $\tilde A\cap \mathbb K(K)=0$. 
	
	$(ii)$ Every operator algebra has a faithful essential representation: take any faithful representation $\sigma: A\to \mathbb B(H)$ and let $\pi=\sigma^{(\infty)}: A\to \mathbb B(H^{(\infty)})$ be direct sum of countably many copies of $\sigma$, then $\pi$ is faithful and essential. 
\end{remark}

\begin{remark}\label{qrfd}
$(i)$ Theorem \ref{rfd} provides a non selfadjoint analog of a result of Goodearl and Menal \cite{gm} as follows: 	Every unital, separable operator algebra $A\subseteq \mathbb B(H)$ whose cone $CA:=C_0(0,1]\otimes A$ is essentially embedded in $\mathbb B(L^2(0,1]\otimes H)$, is a quotient of a block triangular operator algebra. Indeed, let us consider a C*-algebra $D$ containing $A$ and observe that $CA$ is contained in $CD$, which is  contractible and so QD. Hence, $CA$ is QT, and so is $({\rm id}\otimes \pi)(CA)$, for any
faithful essential representation $\pi : A \to \mathbb B(H)$ (which exists by Remark \ref{essentail}$(ii)$). Since $\pi$ is faithful,  ${\rm id}\otimes \pi$ is a complete isometry, thus $({\rm id}\otimes \pi)(CA)$ is essentially embedded. By Theorem \ref{rfd}, we get  an essentially embedded
RFD operator algebra $B\subseteq \mathbb B(L^2(\Omega)\otimes H)$ with $$({\rm id}\otimes \pi)(CA)+ \mathbb K(K) = B + \mathbb K(K),$$ for $\Omega:=(0,1]$ and $K:=L^2(\Omega)\otimes H$. Identifying $A$ with $\pi(A)$, we may identify $C\pi(A)=({\rm id}\otimes \pi)(CA)$ with $CA$.  Passing to the Calkin algebra and using the second equality and the fact that both $B$ and $CA$ are essentially embedded, we observe that 
$$CA/(CA\cap \mathbb K(K))\simeq B/(B\cap \mathbb K(K)),$$ complete isometrically. Finally, $$A\simeq CA/SA\simeq\frac{CA/(CA\cap \mathbb K(K))}{(SA+CA\cap \mathbb K(K))/(CA\cap \mathbb K(K))},$$ 
where the last (complete) isometry follows from the  isomorphism theorems for operator algebras \cite[Theorems 3.1.4, 3.1.5]{al} and Remark \ref{rk}$(iv)$ applied to $I:=SA$ and $J:=CA\cap \mathbb K(K)$ (note that $I$ has a cai, since $A$ is unital). Now the completely isometric homomorphism from $CA/CA\cap \mathbb K(K)$ onto $B/B\cap \mathbb K(K)$ maps $(SA+CA\cap \mathbb K(K))/CA\cap \mathbb K(K)$ to a closed ideal $L/B\cap \mathbb K(K)$, for a closed ideal $L$ of $B$ containing $B\cap \mathbb K(K)$, and we have completely isometric isomorphisms 
$$\frac{CA/CA\cap \mathbb K(K)}{(SA+CA\cap \mathbb K(K))/CA\cap \mathbb K(K)}\simeq \frac{B/B\cap \mathbb K(K)}{L/B\cap \mathbb K(K)}\simeq B/L,$$
by the above mentioned isomorphism theorems. Summing up, 
$A$ is (isometrically isomorphic to) a quotient of the block triangular operator algebra $B$.

$(ii)$ Note that the condition of $A$ or $CA$ being essentially embedded is automatic in the selfadjoint case. The assumption of the cone being essentially embedded (which is automatic for C*-algebras) seems to be crucial in the non selfadjoint case (at least in the above proof). However, at this point we don't know examples of non selfadjoint operator algebras with an essentially embedded cone (except trivial examples, such as finite dimensional triangular matrix algebras). 

$(iii)$ The above cited result of Goodearl and Menal fails when $A$ is not separable (even if $H$ is separable).  Indeed Larry Brown has used Zorn lemma to construct maximal quasidiagonal subsets of $\mathbb B(H)$  and Nate Brown have used this to give counterexamples even in the selfadjoint case (c.f. \cite[Remark 3.7]{b}).

$(iv)$  Finally,  the unitality assumption could be relaxed to approximate unitality, i.e., the existence of a net
consisting of real positive elements in $A$ which is a contractive
approximate identity in any C*-algebra which is generated by $A$ as a closed Jordan subalgebra \cite{ble}.  

 $(v)$ By the isomorphism $A\simeq CA/SA$, every operator algebra is a quotient of a QT operator algebra and the quotient could be explicitly given. Also, for any operator algebra $A\subseteq \mathbb B(H)$ with $H$ separable, since $\mathbb B(H)$ is a quotient of an RFD C*-algebra \cite[Excercise 7.1.4]{bo}, every separable operator algebra is a quotient of an RFD operator algebra. We have just proved a stronger result for separable, essentially embedded operator algebras.

\end{remark}

\begin{proof}[Proof of Proposition~\ref{tr}]  By the proof of Theorem \ref{qtr},  we can choose a sequence of unital
complete contractions $\varphi_k: A \to M_{n(k)}(\mathbb C)$ which form an approximate isometric sequence of approximate homomorphisms. Take a WOT-cluster point $\tau$ of $tr_{n(k)}\circ\varphi_k$ and observe that $\tau$ is a state. The tracial property follows from the fact that  the sequence $(\varphi_k)$ is approximately multiplicative.\end{proof}

Note that even nice unital QD $C^*$-algebras (like the unitization
of the compact operators) fail to have a faithful tracial state. On the other hand, the situation could be saved in the simple case. An operator algebra  is {\it simple} if it has no proper closed ideal. A state $\tau$ on an operator algebra $A$ is faithful if $x=0$ for each $x\in A_{+}:=A\cap C^*(A)_{+}$ with $\tau(x)=0$, where $C^*(A)$ is the $C^*$-algebra of generated by $A$ in a faithful representation.	Now every simple unital operator algebra which is contained in a QD C*-algebra has a faithful trace: First note that since QD passes to subalgebras \cite[Proposition 7.1.10]{bo}, the last assumption is the same as requiring that $C^*(A)$ is QD. Choose a tracial state $\tau$ on $C^*(A)$ by \cite[Proposition 7.1.16]{bo}. We use simplicity to show that the restriction of $\tau$ to $A$ is faithful. Let $$I_{ \tau}:=\{x\in C^*(A): \tau(xx^*)=0\}$$ be the ideal kernel of $\tau$. Take $a\in A_{+}:=A\cap C^*(A)_{+}$ with $\tau(a)=0$ and choose $y\in C^*(A)$ with $a=yy^*$. Then $\tau(yy^*)=\tau(a)=0$ and so $y\in I_{\tau}$. But this is an ideal, thus $a=yy^*\in I_{\tau}$. If $a$ is not zero, then we have a nonzero closed ideal $A\cap I_{\tilde\tau}$ of $A$, which by simplicity has to be equal to $A$, which means that $A\subseteq I_{\tau}$. But then $A^*:=\{x^*: x\in A\}\subseteq I_{\tau}$, leading to the equality $I_{\tau}=C^*(A)$, which is absurd, since $A$ is unital and $\tau(1)=1$.

\begin{remark}
$(i)$ The assumption that $A$ is contained in a QD C*-algebra is stronger than assuming that $A$ is itself QT (see Remark \ref{bdry}$(ii)$).

$(ii)$ It doesn't seem plausible to prove the existence of a faithful trace on a simple unital QT operator algebra, using a modification of the above argument: if we start with a tracial state $\tau$ on $A$ (which exists by roposition~\ref{tr} when $A$ is QT), it is not possible in general to  extend it to a trace on $C^*(A)$. 

Indeed, let $A$ be the norm-closed algebra generated by the unilateral shift in $\mathbb B(\ell^2(\mathbb N\cup\{0\}))$. This is completely isometrically isomorphic the disc algebra, which is in turn the (universal) operator algebra of a contraction. 
Since the disc algebra is abelian and its character space is the closed unit disc, all its characters are tracial states. 

On the other hand, if we have a tracial state $\tau$ on the Toeplitz algebra (the $C^*$-algebra generated by the unilateral shift) then it should annihilate the compacts: Let $p_n$ be the orthogonal projection on $\ell^2$ projecting  to the $n$-th coordinate. Then,
$ \tau(VV^*) = \tau(V^*V) = 1$, and so $\tau(p_0) = \tau(I - VV^*) = 0$. Also, 
$\tau(V^2 (V^2)^*) = 1$,  and so $\tau(p_0 + p_1) = \tau(I - V^2 (V^2)^*) = 0$, that is, $\tau(p_1) = 0$. By an inductive argument, we get $\tau(p_n) = 0$, for each $n$. Now consider  the $(n,n)$-compression state $\tau_n$ on the Toeplitz algebra, that is, $\tau_n(x) = p_n x p_n$. This state does not annihilate the compacts, and by the above argument, its restriction to  $A$ gives  zero Fourier scalars (note that, every element in $A$ admits a representation as an analytic function on the open unit disc), and so is tracial.
\end{remark}

\section*{acknowledgement} The authors would like to thank David Blecher,  Ken Davidson and Evgenios Kakariadis  for suggesting  (non) examples and references. The authors also thank the anonymous referee for careful reading of the paper and suggestions to improve the paper.


\end{document}